\documentclass[12pt,a4paper]{amsart}
\usepackage{amscd,amsmath,amsthm,amssymb}
\usepackage{tikz}
\usepackage{amsmath, amssymb}
\usepackage{array}
\usepackage{float}
\usepackage{wrapfig}
\usepackage{booktabs}
\usepackage{ascmac}
\usepackage{fancybox}
\usepackage{tikz}
\usepackage{bm}
\usepackage{pxpgfmark}
\usepackage{ascmac}
\usepackage[all]{xy}

\usepackage{tikz}
\usetikzlibrary{patterns}
\usetikzlibrary{intersections, calc}
\usetikzlibrary{quotes,angles}
\usetikzlibrary {arrows.meta}
\usetikzlibrary {bending}
\makeatletter
\newcommand{\xequal}[2][]{\ext@arrow 0055{\equalfill@}{#1}{#2}}
\def\equalfill@{\arrowfill@\Relbar\Relbar\Relbar}
\makeatother
\newcommand{\maru}[1]{\raise0.2ex\hbox{\textcircled{\scriptsize{#1}}}}

\def\opn#1#2{\def#1{\operatorname{#2}}} 
	\opn\chara{char} \opn\length{\ell} \opn\pd{pd} \opn\rk{rk}
	\opn\projdim{proj\,dim} \opn\injdim{inj\,dim} \opn\rank{rank}
	\opn\depth{depth} \opn\grade{grade} \opn\height{height}
	\opn\embdim{emb\,dim} \opn\codim{codim}
	\opn\Cl{Cl}
	
	\opn\Tr{Tr} \opn\bigrank{big\,rank}
	\opn\superheight{superheight}\opn\lcm{lcm}
	\opn\trdeg{tr\,deg}
	\opn\rdeg{rdeg}
	\opn\reg{reg} \opn\lreg{lreg} \opn\ini{in} \opn\lpd{lpd}
	\opn\size{size} \opn\sdepth{sdepth}
	\opn\link{link}\opn\fdepth{fdepth}\opn\lex{lex}
	\opn\tr{tr}
	\opn\type{type}
	\opn\gap{gap}
	\opn\arithdeg{arith-deg}
	\opn\revlex{revlex}
	\opn\div{div} \opn\Div{Div} \opn\cl{cl} \opn\Cl{Cl}
	\opn\Spec{Spec} \opn\Supp{Supp} \opn\supp{supp} \opn\Sing{Sing}
	\opn\Ass{Ass} \opn\Min{Min}\opn\Mon{Mon}
	\opn\Ann{Ann} \opn\Rad{Rad} \opn\Soc{Soc}
	\opn\Im{Im} \opn\Ker{Ker} \opn\Coker{Coker} \opn\Am{Am}
	\opn\Hom{Hom} \opn\Tor{Tor} \opn\Ext{Ext} \opn\End{End}
	\opn\Aut{Aut} \opn\id{id}
	
	\opn\nat{nat}
	\opn\pff{pf}
	\opn\Pf{Pf} \opn\GL{GL} \opn\SL{SL} \opn\mod{mod} \opn\ord{ord}
	\opn\Gin{Gin} \opn\Hilb{Hilb}\opn\sort{sort}
	\opn\PF{PF}\opn\Ap{Ap}
	\opn\mult{mult}
	\opn\bight{bight}
	\opn\div{div}
	\opn\Div{Div}
	\opn\aff{aff}
	\opn\relint{relint} \opn\st{st}
	\opn\lk{lk} \opn\cn{cn} \opn\core{core} \opn\vol{vol}  \opn\inp{inp} 
	\opn\nilpot{nilpot}
	\opn\link{link} \opn\star{star}\opn\lex{lex}\opn\set{set}
	\opn\width{wd}
	\opn\Fr{F}
	\opn\QF{QF}
	\opn\G{G}
	\opn\type{type}\opn\res{res}
	\opn\conv{conv}
	\opn\Int{Int}
	\opn\Deg{Deg}
	\opn\Sym{Sym}
	\opn\Con{Con}
	\opn\gr{gr}
	
	\def\pot#1#2{#1[\kern-0.28ex[#2]\kern-0.28ex]}

	\opn\dirlim{\underrightarrow{\lim}}
	\opn\inivlim{\underleftarrow{\lim}}
	\def\Implies{\ifmmode\Longrightarrow \else
		\unskip${}\Longrightarrow{}$\ignorespaces\fi}
	\def\implies{\ifmmode\Rightarrow \else
		\unskip${}\Rightarrow{}$\ignorespaces\fi}
	\def\iff{\ifmmode\Longleftrightarrow \else
		\unskip${}\Longleftrightarrow{}$\ignorespaces\fi}

	\let\:=\colon
	\newtheorem{Theorem}{Theorem}[section]
	\newtheorem{Lemma}[Theorem]{Lemma}

	\theoremstyle{definition}

	\newtheorem{Example}[Theorem]{Example}

	\let\epsilon\varepsilon
	\let\kappa=\varkappa
	\opn\dis{dis}
	\def\pnt{{\raise0.5mm\hbox{\large\bf.}}}
	
	\opn\Lex{Lex}
	
\begin{document}

\title[van der Waerden complexes with linear resolution]{A classification of van der Waerden complexes with linear resolution}
\thanks{Version: March 5, 2025}

\author[T. Hibi]{Takayuki Hibi}
\address{Department of Pure and Applied Mathematics, Graduate School
of Information Science and Technology, Osaka University, Suita, Osaka
565-0871, Japan}
\email{hibi@math.sci.osaka-u.ac.jp}

\author[A. Van Tuyl]{Adam Van Tuyl}
\address{Department of Mathematics and Statistics\\
McMaster University, Hamilton, ON, L8S 4L8, Canada}
\email{vantuyla@mcmaster.ca}

 
\begin{abstract}
In 2017, Ehrenborg, Govindaiah, Park, 
and Readdy defined the
van der Waerden complex ${\tt vdW}(n,k)$ to be the simplicial
complex whose facets correspond to all
the arithmetic sequences on the set $\{1,\ldots,n\}$ 
of a fixed length $k$.  To complement a classification of the Cohen--Macaulay van der Waerden complexes obtained by Hooper and Van Tuyl in 2019, a classification of van der Waerden complexes with linear resolution is presented.  Furthermore, we show that the Stanley--Reisner ring of a Cohen--Macaulay van der Waerden complex is level.
\end{abstract}

\maketitle
\thispagestyle{empty} 

\section*{Introduction}
Fix integers $0 < k < n$.  The {\it van der Waerden}
complex ${\tt vdW}(n,k)$ is the simplicial
complex on the vertex set $\{1,\ldots,n\}$
whose facets are all the arithmetic 
sequences on $\{1,\ldots,n\}$ of length $k$.  As
an example
$${\tt vdW}(7,3) = \langle 
\{1,2,3,4\}, \{2,3,4,5\},\{3,4,5,6\},\{4,5,6,7\},
\{1,3,5,7\} \rangle.$$
Note the length $k$ refers to the number of steps
in the sequence.  By definition,
${\tt vdW}(n,k)$ is a pure simplicial complex
with $\dim {\tt vdW}(n,k) = k.$
  The name
of the complex is inspired by 
van der Waerden's work \cite{VDW1927} on Ramsey Theory that states
for any given two positive integers $k$ and $r$, there 
is an integer $M = M(k,r)$ such that for any $n \geq M$,
if the integers $\{1,\ldots,n\}$ are coloured with
$r$ colours, then there is an arithmetic sequence of length
$k$ that is mono-coloured.

The complexes ${\tt vdW}(n,k)$  were
introduced in 2017 by Ehrenborg, Govindaiah, Park and 
Readdy \cite{EGPR2017} who studied the topology
of these complexes.  Interestingly, the size
of the homotopy type of these complexes is bounded
in terms of the number of distinct prime factors 
of all integers less than or equal to $k$, making
these complexes of interest in number theory. 
Hooper and Van Tuyl \cite{HVT2019} later classified
for what values $n$ and $k$ the complex ${\tt vdW}(n,k)$
is a vertex decomposable, shellable, or 
Cohen--Macaulay simplicial complex (see Theorem \ref{AVT}).  
In fact,
${\tt vdW}(n,k)$ either has all three properties,
or it has none of these properties.

The purpose of the present paper is to revisit
the classification of Hooper and Van Tuyl \cite{HVT2019}
and to complement their result with considering linear resolutions and level rings.  In particular, we classify  
the ${\tt vdW}(n,k)$ whose Stanley--Reisner
ideals have linear resolution (Theorem \ref{linear}) and show that the Stanley--Reisner ring of a Cohen--Macaulay van der Waerden complex is level (Theorem \ref{level}). 

\section{Linear resolution}
A classification of van der Waerden complexes with linear resolution is presented.  First we recall fundamental materials from \cite{HHgtm260}.  

Let $\Delta$ be a simplicial complex.  A vertex $x$ of $\Delta$ is called {\em free} if $x$ belongs to exactly one facet.  A facet $F$ of $\Delta$ is said to be a {\em leaf} of $\Delta$ if there exists a facet $G$ of $\Delta$ with $G \neq F$, called a {\em branch} of $F$, for which $H \cap F \subset G \cap F$ for all facets $H$ with $H \neq F$.  Each vertex belonging to $F \setminus G$ is free.  A labelling $F_1, \ldots, F_s$ of the facets of $\Delta$ is called a {\em leaf order} if, for each $1 < i \leq s$, the facet $F_i$ is a leaf of the subcomplex $\langle F_1, \ldots, F_{i}\rangle$.  A simplicial complex that has a leaf order is called a {\em quasi-forest}.  Recall from \cite[Theorem 9.2.12]{HHgtm260} that a simplicil complex $\Delta$ is a quai-forest if and only if $\Delta$ is the clique complex of a chordal graph.  Furthermore, it is known by work of Fr\"oberg (e.g., see \cite[Theorem 9.2.3]{HHgtm260}) that the Stanley--Reisner ideal of the clique complex of a finite graph $G$ has linear resolution if and only if $G$ is a chordal graph.   We say an ideal $I$ has {\it linear 
resolution} if all the generators of $I$ have the 
same degree $d \geq 1$ for some integer $d$
and $\beta_{i,i+j}(I) = 0$ for all $j \neq d$.

\begin{Theorem}
\label{linear}
Let $0 < k < n$ be integers.  Then the Stanley--Reisner ideal of ${\tt vdW}(n,k)$ has linear resolution if and only if either $k=1$ or $\frac{n}{2} \leq k < n$.    
\end{Theorem}

\begin{proof}
{\bf (If)} Let $k = 1$.  Then ${\tt vdW}(n,1)$ is the complete graph on $n$ vertices.  Thus its Stanley--Reisner ideal is generated by all monomials of degree $3$.  Hence it has a linear resolution (\cite[Problem 8.7]{HHgtm260}). 

Let $\frac{n}{2} \leq k < n$.  Then the facets of ${\tt vdW}(n,k)$ are $F_i=\{i, i+1, \ldots, i+k\}$ with $1 \leq i < i+k \leq n$.  It then follows that the ordering $F_1, \ldots, F_{n-k}$ is a leaf order and that ${\tt vdW}(n,k)$ is a quasi-forest.  Hence the Stanley--Reisner ideal of ${\tt vdW}(n,k)$ has a linear resolution.

{\bf (Only If)} Let $1 < k < \frac{n}{2}$ with $n \geq 7$.  Lemma \ref{minimalnonfaces} below guarantees that the Stanley--Reisner ideal of ${\tt vdW}(n,k)$ does not have linear resolution.  Finally, since each of ${\tt vdW}(6,2)$ and ${\tt vdW}(5,2)$ (see
Figure \ref{fig.vdw5262}) is the clique complex of a non-chordal graph, its Stanley--Reisner ideal does not have linear resolution
(also Example \ref{ex}).
\end{proof}

\begin{figure}[ht]
\begin{tikzpicture}[scale=1.5]
\coordinate (1) at (0,2);
\coordinate (2) at (2,2);
\coordinate (3) at (1,1);
\coordinate (4) at (2,0);
\coordinate (5) at (0,0);
\fill(1)circle(0.7mm);
\fill(2)circle(0.7mm);
\fill(3)circle(0.7mm);
\fill(4)circle(0.7mm);
\fill(5)circle(0.7mm);
\draw (1)node[above=1mm]{$1$}--(2)node[above=1mm]{$2$}--(4)node[below=1mm]{$4$}--(5)node[below=1mm]{$5$}--cycle;
\draw (1)--(3)node[above=1mm]{$3$}--(4);
\draw(2)--(3)--(5);
\end{tikzpicture}
\hspace{3cm}
\begin{tikzpicture}[scale=1.5]
\coordinate (2) at (0,2);
\coordinate (6) at (2,2);
\coordinate (5) at (2,0);
\coordinate (1) at (0,0);
\coordinate (3) at ($(1)!1/3!(6)$);
\coordinate (4) at ($(1)!2/3!(6)$);
\fill(1)circle(0.7mm);
\fill(2)circle(0.7mm);
\fill(3)circle(0.7mm);
\fill(5)circle(0.7mm);
\fill(6)circle(0.7mm);
\fill(4)circle(0.7mm);
\draw (1)node[below=1mm]{$1$}--(2)node[above=1mm]{$2$}--(6)node[above=1mm]{$6$}--(5)node[below=1mm]{$5$}--cycle;
\draw (1)--(3)node[below=1mm]{$3$}--(4)node[above=1mm]{$4$}--(6);
\draw(2)--(3)--(5);
\draw(2)--(4)--(5);
\end{tikzpicture}
\caption{${\tt vdW}(5,2)$ and ${\tt vdW}(6,2)$}\label{fig.vdw5262}
\end{figure}
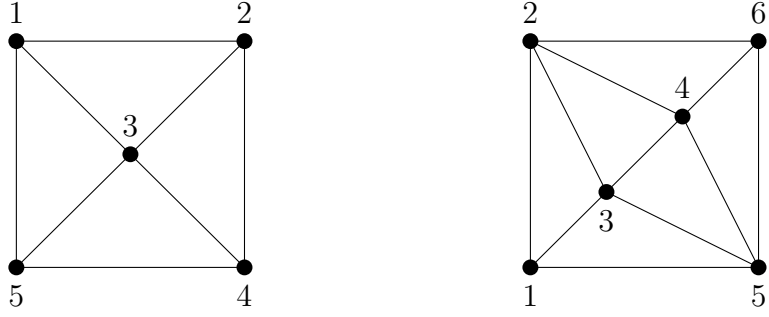

\begin{Lemma}\label{minimalnonfaces}
    Let $1 < k < \frac{n}{2}$ with $n
    \geq 7$ and let $d$ be the largest integer such
    that $1+kd \leq n$. Then the simplicial
    complex ${\tt vdW}(n,k)$ has the following
    minimal non-faces:
    \begin{enumerate}
        \item $\{1,kd\}$; and 
        \item $\{1,1+k(d-1),1+kd\}$ if $d \nmid k$.
        \item $\{1,1+(k-1)(d-1),1+(k-1)d\}$ if $d|k$ and $d <k$.
        \item $\{1,1+(k-2)d, 1+(k-1)(d-1)\}$ if $d=k$.
    \end{enumerate}
Consequently, the Stanley--Reisner
ideal of ${\tt vdW}(n,k)$ has
minimal generators of degree two and three.
\end{Lemma}

\begin{proof}
    The final statement follows from the definition of a Stanley--Reisner
    ideal.   
    
    Observe that if $d$ is the largest integer
    such that $1+kd\leq n$, then 
    $2 \leq d$ since $k < \frac{n}{2}$ 
    implies $2k < n$, or equivalently, 
    $1+2k \leq n$.  We introduce the following notation:
    \begin{equation*}\label{facetswithone}
    F_j = \{1,1+1\cdot j, 1+2\cdot j,\ldots, 1+k\cdot j\}
    ~~\mbox{for $j=1\ldots,d$.}
    \end{equation*}
    The sets $F_1,\ldots,F_d$ are all the facets of ${\tt vdW}(n,k)$
    that contain $1$.  Note that if $x,y \in F_j$, then
    $|x-y|$ is a multiple of $j$.
    
    $(1)$ 
    Since $d \geq 2$, we have $1+(k-1)d < kd
    < 1+kd$.  So $kd$ does not appear
    in the the facet $F_d$.  Furthermore,
    we have $1+kj < dk$ for all $j=1,\ldots,
    d-1$.  Indeed, if $1+kj \geq kd$ for
    some $j$, then $1 \geq k(d-j)$, with
    $d-j \geq 1$ and $k>1$, a contradiction.
    So $\{1,kd\}$ does not appear in  any
    facet of ${\tt vdW}(n,k)$.  It is clear
    that this set is a minimal non-face.

$(2)$ Suppose that $d\nmid k$.  The only facet
that contains $\{1,1+kd\}$ is $F_d$.  Moreover, $1+k(d-1) \not\in F_d$.
If it was, then $|(1+k(d-1))-(1+kd)|=k$ is a multiple of $d$, which
is not true.   So $\{1,1+k(d-1),1+dk\}$ is a non-face.  As we already
observed, $\{1,1+kd\}$ is a subset of $F_d$, and $\{1,1+k(d-1)\}$
is a subset of $F_{d-1}$.  Finally, $\{1+k(d-1), 1+kd\}$ belongs
to the facet
$$\{1+k(d-1),1+k(d-1)+1,\ldots,1+k(d-1)+k\}$$
since $1+k(d-1)+k = 1+kd$.  So $\{1,1+k(d-1),1+kd\}$ is a minimal
non-face of ${\tt vdW}(n,k)$.

$(3)$ Suppose that $d\mid k$ and $d<k$.  We have $\{1,1+(k-1)d\}$ 
as a subset of $F_d$ and $\{1,1+(k-1)(d-1)\}$ a subset of $F_{d-1}$.
Also, $\{1+(k-1)(d-1),1+(k-1)d\}$ is an element of the
facet
$$\{1+(k-1)(d-1),1+(k-1)(d-1)+1,\ldots,1+(k-1)(d-1)+(k-1),1+(k-1)(d-1)+k\}.$$
Indeed, the second last element of this set is $1+(k-1)d$.

Since $d < k$, we have 
$$1+(k-1)(d-1) < 1+k(d-1) <  1+kd-d = 1+(k-1)d.$$
So $1+(k-1)d$ only appears in the facet $F_d$, but not in
$F_1,\ldots,F_{d-1}$. Indeed, $1+k(d-1)$ is the 
largest element of $F_{d-1}$. On the other hand,
$[1+(k-1)d]-[1+(k-1)(d-1)]=k-1$.  So, if
$1+(k-1)(d-1)$ was in $F_d$,
we would have $d|(k-1)$.  But $d|k$, and thus $d|(k-(k-1))=1$.
This implies $2 \leq d <1$, a contradiction.
So $\{1,1+(k-1)(d-1),1+(k-1)d\}$ is a minimal non-face, as desired.

$(4)$ Suppose that $d=k$.  In this case $\{1,1+(k-2)d\}$ is a
subset of $F_d$ and $\{1,1+(k-1)(d-1)\}$ is a subset of $F_{d-1}$.
We also have $\{1+(k-1)(d-1),1+(k-2)d\}$ is a subset 
of the facet
$$\{1+(k-2)d,1+(k-2)d+1,1+(k-2)d+2,\ldots,1+(k-2)d+k\}$$
Indeed, the term $1+(k-2)d+1$ that appear in this facet is equal
to $1+dk-2d+1 = 1+k^2-2k+1 = 1+ (k-1)(k-1) = 1+(k-1)(d-1)$, where
we are using the fact that $k=d$.    

Finally, we need to
show that $\{1,1+(k-2)d,1+(k-1)(d-1)\}$ is a non-face. 
Since $k=d$,  one has $|(1+(k-2)d) - (1+(k-1)(d-1))| = |(1+(d-2)d)-(1+(d-1)(d-1)|
= 1$.  Thus if $\{1,1+(k-2)d,1+(k-1)(d-1)\}$ is a face, then it is a subset of $F_1$.  Since $1+(k-1)(d-1) \geq 1 + (k-1) = k$, if it is a subset of $F_1$, then $k = d = 2$.  Since $n \geq 7$, if $k=2$, then $d = 3$.  Hence $k = d = 2$ cannot happen.
\end{proof}

\section{Level rings}
The following classification is \cite[Theorem 1.1]{HVT2019}.  
\begin{Theorem} 
\label{AVT}
Fix integers $0 < k < n$.  Then  the following
are equivalent:
\begin{enumerate}
    \item $n \leq 6$, or $n \geq 7$ and $k=1$ or
    $\frac{n}{2} \leq k < n$; 
    \item ${\tt vdW}(n,k)$ is vertex decomposable;
    \item ${\tt vdW}(n,k)$ is shellable;
    \item ${\tt vdW}(n,k)$ is Cohen--Macaulay;
\end{enumerate}
\end{Theorem}

We now come to show the Stanley--Reisner rings of all 
Cohen--Macaulay van der
Waerden complexes are level. 
Recall that we say that $R/I$ is {\it level} 
if $R/I$ is Cohen--Macaulay
and the last column of the Betti table of $R/I$ contains a
single entry. The level ring $R/I$ is {\it Gorenstein} 
if the only entry in
the last column is $1$.

\begin{Example}\label{ex}
    If $I$ is the Stanley--Reisner ideal of ${\tt vdW}(5,2)$,
    its Betti table is
    $$\begin{matrix}
 & 0 & 1 & 2\\
\text{total:} & 1 & 2 & 1\\
0: & 1 & . & .\\
1: & . & 2 & .\\
2: & . & . & 1
\end{matrix}$$
For the van der Waerden complex ${\tt vdW}(6,2)$ the 
Betti table of the Stanley--Reisner ideal is
$$\begin{matrix}
 & 0 & 1 & 2 & 3\\
\text{total:} & 1 & 4 & 5 & 2\\
0: & 1 & . & . & .\\
1: & . & 4 & 2 & .\\
2: & . & . & 3 & 2
\end{matrix}$$
We see that neither ideal has linear resolution.   However,
both are level, and ${\tt vdW}(5,2)$ is Gorenstein. 
 In
fact, since the Stanley--Reisner ideal of ${\tt vdW}(5,2)$ is
$\langle x_1x_4, x_2x_5 \rangle$ in $k[x_1,\ldots,x_5]$,
the Stanley--Reisner ideal is a complete intersection.
 Note that these Betti tables 
show that Lemma \ref{minimalnonfaces} cannot be extended to $n <7$ since both of
these ideals are only generated in degree 2.
\end{Example}

\begin{Theorem}
\label{level}
Suppose that ${\tt vdW}(n,k)$ is Cohen--Macaulay.  Then, the Stanley--Reisner ring of  ${\tt vdW}(n,k)$ is level.  Furthermore, ${\tt vdW}(n,k)$ is Gorenstein if and only if $(n,k) = (5,2).$  
\end{Theorem} 

\begin{proof}
Suppose that ${\tt vdwW}(n,k)$ is Cohen--Macaulay.
 By Theorems \ref{linear} and \ref{AVT}, except for 
 ${\tt vdW}(5,2)$ and ${\tt vdW}(6,2)$, the Stanley--Reisner ring of ${\tt vdW}(n,k)$ has linear resolution.  In particular it is level, but not Gorenstein (the Betti table
 of a Gorenstein ring must be symmetric, so it
 cannot have a linear resolution). On the other hand, it can be shown by directly
 computing the Betti diagrams
 as in Example \ref{ex} that ${\tt vdW}(5,2)$ is Gorenstein and ${\tt vdW}(6,2)$ is level, but not Gorenstein.  
\end{proof}

\subsection*{Acknowledgments}
Van Tuyl’s research is supported by NSERC Discovery Grant 2024-05299.
\bibliographystyle{amsplain}
\bibliography{bibliography}

\end{document}